\documentclass[11pt,reqno,oneside]{amsart}
\usepackage{graphics}
\usepackage[dviwin]{color}

\newtheorem{theorem}{Theorem}[section]
\newtheorem{lemma}{Lemma}[section]
\newtheorem{corollary}{Corollary}[section]
\newtheorem{remark}{Remark}[section]

\def\U{{\mathrm{U}}}
\def\o{\Omega}
\def\O{\Omega}
\def\w{\omega}
\def\ho{\widehat{\Omega}}
\def\oj{\o^{n^\prime,s}}
\def\hoj{\ho^{n^\prime,s}}
\def\g{\gamma}
\def\a{\alpha}
\def\b{\beta}
\def\zp{z^\prime}
\def\hzp{\hz^\prime}
\def\hx{\hat{x}}
\def\hy{\hat{y}}
\def\hz{\hat{z}}

\def\hg{\hat{g}}
\def\hb{\hat{\beta}}

\def\uu{{\bf u}}

\def\vv{{\bf v}}
\def\ff{{\bf g}}
\def\ww{{\bf w}}

\def\G{{\bf G}}

\def\R{{\mathbb R}}

\def\NN{{\mathbb N}}

\def\W1p{W^{1,p}(\o)}

\def\di{\mbox{div\,}}

\def\ve{\varepsilon}

\begin{document}

\section*{}

\setcounter{equation}{0}
\title[]{Solutions of the divergence and Korn inequalities on domains with an
external cusp}

\author[R. G. Dur\'an]{Ricardo G. Dur\'an}
\address{Departamento de Matem\'atica\\ Facultad de
Ciencias Exactas y Naturales\\Universidad de Buenos Aires\\1428 Buenos Aires\\
Argentina.} \email{rduran@dm.uba.ar}

\author[F. L\'opez Garc\'\i a]{Fernando L\'opez Garc\'\i a}
\address{Departamento de Matem\'atica\\ Facultad de
Ciencias Exactas y Naturales\\Universidad de Buenos Aires\\1428 Buenos Aires\\
Argentina.} \email{flopezg@dm.uba.ar}

\thanks{Supported by ANPCyT under grant PICT 2006-01307, by Universidad de Buenos Aires under
grant X070 and by CONICET under grant PIP 5478. The first
author is a member of CONICET, Argentina.}

\keywords{Divergence Operator, weighted Sobolev spaces, Korn
inequality }

\subjclass{Primary: 26D10 , 35Q30 ; Secondary 76D03}

\begin{abstract} This paper deals with solutions of the divergence
for domains with external cusps. It is known that the classic
results in standard Sobolev spaces, which are basic in the
variational analysis of the Stokes equations, are not valid for
this class of domains.

For some bounded domains $\Omega\subset{\mathbb R}^n$ presenting  power type
cusps of integer dimension $m\le n-2$, we prove the existence of
solutions of the equation $\mbox{div\,}{\bf u}=f$ in weighted Sobolev spaces,
where the weights are powers of the distance to the cusp. The
results obtained are optimal in the sense that the powers cannot
be improved.

As an application, we prove existence and uniqueness of solutions
of the Stokes equations in appropriate spaces for cuspidal
domains. Also, we obtain weighted Korn type inequalities for this
class of domains.
\end{abstract}
\maketitle

\section{Introduction}

This paper deals with solutions of the divergence in domains with
external cusps. Given a bounded domain $\O\subset\R^n$, it is
known that, under appropriate assumptions on $\O$, there exists a
continuous right inverse of the operator ${\rm
div}:W^{1,p}_0(\O)^n\to L_0^p(\O)$, $1<p<\infty$, where
$L_0^p(\O)$ denotes the space of functions in $L^p(\O)$ with
vanishing mean value in $\O$. In other words, given any $f\in
L_0^p(\O)$, there exists a solution $\uu\in W^{1,p}_0(\O)^n$ of
\begin{eqnarray}
\label{div} \di \uu=f
\end{eqnarray}
satisfying
\begin{eqnarray}
\label{norma} \|\uu\|_{W^{1,p}(\O)} \le C\|f\|_{L^p(\O)},
\end{eqnarray}
where the constant $C$ depends only on $\O$ and $p$.

This result has many applications, for example, in the particular
case $p=2$, it is a basic tool for the variational analysis of the
Stokes equations and it implies the Korn inequality in its more
general form (see for example \cite{BS,C}).

Consequently, this problem has been widely studied and several
methods to prove the existence of $\uu$ satisfying (\ref{div}) and
(\ref{norma}), under different assumptions on the domain, have
been developed (see for example \cite{ADM,ASV,B,BS,DRS,GR,BA,L}).

On the other hand, it is known that this result does not hold for
domains with external cusps. Several arguments have been given to
show this fact \cite{ADLg,D,GG}, but the oldest counterexample
goes back to Friedrichs, who showed that an inequality for
analytic complex functions (which follows easily from the
existence of $\uu$ satisfying (\ref{div}) and (\ref{norma})) does
not hold in a domain with a quadratic external cusp (see
\cite{F}).

Therefore, it is an interesting question what kind of weaker
results can be proved for these domains and whether these results
can be applied to show the well posedness of the Stokes equations
in appropriate spaces. Since the problem arises because of the bad
behavior of the boundary, it seems natural to work with weighted
Sobolev spaces where the weights are related to the distance to
the boundary or to its singularities.

Recently, in \cite{DLg}, we have obtained results for planar
simply connected H\"older-$\alpha$ domains working with weights
which are powers of the distance to the boundary of $\O$. The
domains with external cusps that we are going to consider in this
paper are a subclass of the H\"older-$\alpha$ domains. However,
for this particular subclass, it is natural to look for stronger
results where the distance to the boundary is replaced by the
distance to the cusp, which can be a point or more generally a set
of dimension $m\le n-2$. To obtain this kind of results is the
main goal of this paper.

As mentioned above, an important consequence of the existence of
continuous right inverses of the divergence is the Korn
inequality. We are going to show that the known arguments can be
extended to some weighted cases allowing us to obtain new weighted
Korn inequalities for domains with external cusps.

Our results are optimal in the sense that the powers of the
distance to the cusp involved in the estimates cannot be improved,
this is proved in \cite{ADLg}.

As an application we prove the well posedness in appropriate
spaces of the Stokes equations in domains with external cusps. In
the particular two dimensional case similar results were proved in
our previous paper \cite{DLg} but a restriction in the power of
the cusp was needed (this restriction is removed here).

The rest of the paper is organized as follows. Since the analysis
of the Stokes equations is our main motivation, we start
developing a generalized variational analysis of these equations,
this is done in Section \ref{section2}. Also in this section, we
show by a simple example, that the existence of solution of the
Stokes equations in the standard spaces is not true for cuspidal
domains. In Section \ref{section3} we prove some auxiliary results
concerning solutions of the divergence in weighted Sobolev spaces
for domains which are star-shaped with respect to a ball. Section
\ref{sol de la div} contains our main results, namely, the
existence of solutions of the divergence in appropriate spaces for
cuspidal domains. Finally, Sections \ref{aplicacion1} and
\ref{aplicacion2} deal with the applications to the Stokes
equations and to the Korn inequalities respectively.

We will work with weighted $L^p$-norms. Given a non-negative
function $\w$ and a domain $\O\subset\R^n$ we denote with
$L^p(\O,\w)$ the Banach space with norm given by
$$
\|f\|^p_{L^p(\O,\w)}=\int_\O |f(x)|^p \, \w(x)\, dx.
$$
If $\w$ is such that $L^p(\O,\w)\subset L^1(\O)$, $L_0^p(\O,\w)$
denotes the subspace of $L^p(\O,\w)$ of functions with vanishing
mean value in $\O$.

\section{Generalized variational analysis of the Stokes equations}
\label{section2} \setcounter{equation}{0}

The goal of this section is to explain the motivation of the main
results of this paper, namely, the existence of right inverses of
the divergence in weighted Sobolev spaces.

First of all, we show by a simple example that the Stokes system
of equations is not well posed in the usual Sobolev spaces for
domains with external cusps. In view of this fact we introduce a
generalization of the classic analysis for this kind of domains.
We will use the usual notations for Sobolev spaces.

The Stokes equations are given by

\begin{align}
\label{stokes1}
\begin{cases}
-\Delta\uu\,+\,\nabla p&=\,\ff\hspace*{1cm}{\rm in}\ \o\\
\di\uu&=\,0\hspace*{1cm}{\rm in}\ \o\\
\uu&=\,0\hspace*{1cm}{\rm on}\ \partial\o.
\end{cases}
\end{align}

For a bounded domain $\O$ which is  Lipschitz (or more generally a
John domain \cite{ADM}) it is known that, if $\ff\in
H^{-1}(\O)^n$, then there exists a unique solution
$$
(\uu,p)\in H^1_0(\O)^n\times L^2_0(\O).
$$
Moreover, the following a priori estimate holds
\[
\|\uu\|_{H^1(\O)^n} + \|p\|_{L^2(\O)} \le C
\|\ff\|_{H^{-1}(\O)^n},
\]
where the constant $C$ depends only on the domain $\O$.

Let us show that this result is not valid in general for domains
with external cusps. Consider for example the domain
\begin{equation}
\label{cuspide1} \O=\Big\{x=(x_1,x_2)\in \R^2:\, 0<x_1<1 \, , \,
|x_2|<x_1^2\Big\}.
\end{equation}
It is known that there exists a function $p\in L^1_0(\O)$ such
that $\nabla p \in H^{-1}(\O)^2$ but $p\notin L^2(\O)$. A simple
example given by G. Acosta is $p(x_1,x_2)=\frac{1}{x_1^2}-6$.
Indeed, by elementary integration one can easily check that
$p\notin L^2(\O)$. On the other hand, to see that $\nabla p \in
H^{-1}(\O)^2$, we only have to show that $\frac{\partial
p}{\partial x_1}\in H^{-1}(\O)$, but this follows from
$\frac{\partial p}{\partial x_1} =\frac{\partial}{\partial
x_2}\left(-\frac{2x_2}{x_1^3}\right)$ and $-\frac{2x_2}{x_1^3}\in
L^2(\O)$.

Consider now the Stokes problem (\ref{stokes1}) with \[
\ff(x)=\left(-\frac{2}{x_1^3},0\right)=\nabla p\in H^{-1}(\O)^2.
\]
Then,
\[
(\uu,p)=\left({\bf 0},\frac{1}{x_1^2}-6\right)
\]
is a solution, but $p\notin L^2(\O)$.

One could ask whether another solution in the space
$H_0^1(\O)^2\times L_0^2(\O)$ exists. That this is not the case
will follow from our general results which, for this particular
domain, give existence and uniqueness (up to an additive constant
in the pressure) in the space
\[
H_0^1(\O)^2\times L^2(\O,|x|^2)\supset H_0^1(\O)^2\times L^2(\O)
\]
and it is easy to see that our solution $({\bf 0},p)$ belongs to
this larger space.

Our general existence and uniqueness results for domains with
cusps follow from the classic theory but replacing the usual
Sobolev spaces by appropriate weighted spaces.

The classic analysis of the Stokes equations is based on the
abstract theory for saddle point problems given by Brezzi in
\cite{Br} (see also the books \cite{BF,GR,BDF}).

Indeed, the weak formulation of (\ref{stokes1}) can be written as

\begin{align}
\label{stokes debil}
\begin{cases}
a(\uu,\vv) + b(\vv,p)&= \int_\O\ff\cdot\vv \qquad \forall \vv\in V\\
b(\uu,q)&=0\qquad \qquad\ \ \forall q\in Q,
\end{cases}
\end{align}
where
\[
a(\uu,\vv)=\int_\O D\uu : D\vv
\]
and
\[
b(\vv,p)= \int_\O p\,\di\vv,
\]
where, for $\vv\in H^1(\O)^n$, $D\vv$ is its differential matrix
and, given two matrices $A=(a_{ij})$ and $B=(b_{ij})$ in
$\R^{n\times n}$, $A:B=\sum_{i,j=1}^{n}=a_{ij}b_{ij}$.

The abstract theory gives existence and uniqueness for
(\ref{stokes debil}) when $a$ and $b$ are continuous bilinear
forms, $a$ is coercive on the kernel of the operator $B\,:\,V\to
Q'$ associated with $b$, and $b$ satisfies the inf-sup condition
\[
\inf_{0\neq q\in Q}\sup_{0\neq\vv\in V} \frac{b(\vv,q)}
{\|q\|_Q\|\vv\|_V}>0.
\]
In the case of the Stokes problem, if we choose the spaces
$V=H^1_0(\O)^n$ and $Q=L_0^2(\O)$, continuity of the bilinear
forms and coercivity of $a$ follow immediately by Schwarz and
Poincar\'e inequalities. Therefore, the problem reduces to prove
the inf-sup condition for $b$ which reads

\begin{equation}
\label{inf-sup} \inf_{0\neq q\in L^2_0(\O)}\sup_{0\neq\vv\in
H^1_0(\O)^n} \frac{\int_\O
q\,\di\vv}{\|q\|_{L_0^2(\O)}\|\vv\|_{H^1_0(\O)^n}}>0.
\end{equation}
It is well known that this condition is equivalent to the
existence of solutions of $\di\uu=f$, for any $f\in L_0^2(\O)$,
with $\uu\in H^1_0(\O)^n$ satisfying $\|\uu\|_{H^1(\O)^n}\le
C\|f\|_{L^2(\O)}$.

Observe that in the above example of cuspidal domain, the
condition (\ref{inf-sup}) does not hold, because it would imply
existence of solution $(\uu,p)\in H^1_0(\O)^n\times L_0^2(\O)$ for
any $\ff\in H^{-1}(\O)^n$ and we have shown that this is not
possible.

For domains such that (\ref{inf-sup}) is not valid, the idea is to
replace this condition by a weaker one. With this goal we will
work with weighted norms.

We will use the following facts for $\w\in L^1(\O)$ which are easy
to see. First, $L^2(\O,\w^{-1})\subset L^1(\O)$ and therefore
$L_0^2(\O,\w^{-1})$ is well defined, and second, the integral
$\int_\O q\,\w$ is well defined for $q\in L^2(\O,\w)$ and
therefore we can define the space
\[
L_{\w,0}^2(\O,\w) =\left\{q\in L^2(\O,\w)\,:\, \int_\O
q\,\w=0\right\}.
\]

We have the following generalization of the classic result which
will be useful for cuspidal domains.

\begin{theorem}
\label{existencia y unicidad} Let $\w\in L^1(\O)$ be a positive
weight. Assume that for any $f\in L^2_0(\O,\w^{-1})$  there exists
$\uu\in H_0^1(\O)^n$ such that $\di\uu=f$ and
\[
\|\uu\|_{H^1(\O)^n}\le C_1 \|f\|_{L^2(\O,\w^{-1})},
\]
with a constant $C_1$ depending only on $\O$ and $\w$. Then, for
any $\ff\in H^{-1}(\O)^n$, there exists a unique $(\uu,p)\in
H^1_0(\O)^n\times L_{\w,0}^2(\O,\w)$ solution of the Stokes
problem (\ref{stokes1}). Moreover,
\[
\|\uu\|_{H^1(\O)^n} + \|p\|_{L^2(\O,\w)} \le C_2
\|\ff\|_{H^{-1}(\O)^n},
\]
where $C_2$ depends only on $C_1$ and $\O$.
\end{theorem}

\begin{proof} We apply the general abstract theory for saddle point problems
with appropriate spaces.

For the pressure we introduce the space $Q=L_{\w,0}^2(\O,\w)$ with
the norm $\|q\|_Q=\|q\|_{L^2(\O,\w)}$.

Since we are modifying the pressure space, we have to enlarge the
$H^1$-norm of the velocity space in order to preserve continuity
of the bilinear form $b$. Then, we define
\[
V=\Big\{\vv\in H^1_0(\O)^n\,:\, \di\vv \in L^2(\O,\w^{-1})\Big\}
\]
with the norm given by
\[
\|\vv\|^2_V = \|\vv\|^2_{H^1(\O)^n} +
\|\di\vv\|^2_{L^2(\O,\w^{-1})}.
\]

Since $\|\vv\|_{H^1(\O)^n}\le\|\vv\|_V$ the continuity of $a$ in
$V\times V$ follows immediately by Schwarz inequality. Also, from
the definitions of the spaces it is easy to see that $b$ is
continuous on $V\times Q$.

On the other hand, coercivity of $a$, in the norm of $V$, on the
kernel of the operator $B$ follows from Poincar\'e inequality
because this kernel consists of divergence free vector fields.

Therefore, to apply the general theory it only rests to prove the
inf-sup condition
\begin{equation}
\label{if} \inf_{0\neq q\in Q}\sup_{0\neq\vv\in V} \frac{\int_\O
q\,\di\vv} {\|q\|_Q\|\vv\|_V}>0.
\end{equation}
But this follows in a standard way. Indeed, given $q\in Q$ it
follows from our hypothesis that there exists $\uu\in H_0^1(\O)^n$
such that $\di\uu=q\,\w$ and
\[
\|\uu\|_{H^1(\O)^n}\le C_1 \|q\,\w\|_{L^2(\O,\w^{-1})}
=C_1\|q\|_Q.
\]
Moreover, since $\|\di\uu\|_{L^2(\O,\w^{-1})} =\|q\|_Q$ we have
\[
\|\uu\|_V\le C\|q\|_Q,
\]
with $C$ depending only on $C_1$.

Then,
\[
\|q\|_Q =\frac{\int_\O q\, q\,\w}{\|q\|_Q} \le C \frac{\int_\O q\,
\di\uu}{\|\uu\|_V}
\]
and therefore (\ref{if}) holds.
\end{proof}

As an example let us mention that the hypothesis of the theorem
holds for the case of the cuspidal domain introduced in
(\ref{cuspide1}) with $\w(x)=|x|^2$. This result is a particular
case of the general results that we are going to prove in Section
\ref{sol de la div}. Consequently, there is a unique weak solution
$(\uu,p)\in H_0^1(\O)^2\times L_0^2(\O,|x|^2)$ of the Stokes
equations (\ref{stokes1}) in this domain.

\section{Solutions of the divergence in star-shaped domains}
\label{section3} \setcounter{equation}{0}

We will work with weighted Sobolev spaces. Given weights
$\w_1,\w_2:\R^n\rightarrow [0,\infty]$, for any domain
$U\subset\R^n$ and $1<p<\infty$, we define

\[
W^{1,p}(U,\w_1,\w_2) =\left\{f\in L^p(U,\w_1)\,:\, \frac{\partial
f}{\partial x_i}\in L^p(U,\w_2),\, \ 1\le i\le n\right\}
\]
with the norm given by
\[
\|f\|^p_{W^{1,p}(U,\w_1,\w_2)}=\int_U|f(x)|^p\w_1(x)\,dx
+\sum_{i=1}^n\int_U \left|\frac{\partial f(x)}{\partial
x_i}\right|^p\w_2(x)\,dx.
\]
To simplify notation we will write $W^{1,p}(U,\w)$ instead of
$W^{1,p}(U,\w,\w)$.

To prove our main results concerning solutions of the divergence
in cuspidal domains, we will make use of the existence of
solutions in weighted Sobolev spaces for good domains.

We will work with weights in the Muckenhoupt class $A_p$ (see for
example \cite{Du,S2}). Recall that, for $1<p<\infty$, a
non-negative weight defined in $\R^n$ is in $A_p$ if
\begin{eqnarray*}
\sup_{B\subset\R^n}\left(\frac{1}{|B|}\int_B
\w\right)\left(\frac{1}{|B|}\int_B
\w^{-1/(p-1)}\right)^{p-1}<\infty,
\end{eqnarray*}
where the supremum is taken over all the balls $B\subset\R^n$ and
$|B|$ denotes the Lebesgue measure of $B$. It is known that, if
$\w\in A_p$, the spaces $W^{1,p}(U,\w)$ and $L^p(U,\w)$ are Banach
spaces (see \cite{GU}).

\begin{remark}
\label{integrable} If $U$ is a bounded domain and $\w\in A_p$
then, $L^p(U,\w)\subset L^1(U)$. Indeed, let $B$ a ball containing
$U$. We have,
\begin{eqnarray*}
\int_U|f|&=&\int_U|f|\w^{1/p}\w^{-1/p}\leq
\left(\int_U|f|^p\w\right)^{1/p}\left(\int_U
\w^{-p^\prime/p}\right)^{1/p^\prime}\\
&\leq&|B|^{(p-1)/p}\|f\|_{L^p(U,\w)}\left(\frac{1}{|B|}\int_B
\w^{-1/(p-1)}\right)^{(p-1)/p}.
\end{eqnarray*}
\end{remark}

In view of this remark the space $L_0^p(U,\w)$ is well defined. We
will work also with the space $W_0^{1,p}(U,\w)$ defined as the
closure of $C_0^\infty(U)$ in $W^{1,p}(U,\w)$.

Now we give the auxiliary result that we need. We state it as a
theorem since it can be of interest in itself. We outline a proof
based on Bogovskii's formula for solutions of the divergence
\cite{B,DM,G}. An alternative proof of this result was given in
\cite{DRS}.

\begin{theorem}
\label{divAp} Let $\w\in A_p$, $1<p<\infty$, and $U\subset\R^n$ be
a bounded domain which is star-shaped with respect to a ball
$B\subset U$. Given $f\in L_0^p(U,\w)$, there exists $\uu\in
W^{1,p}_0(U,\w)^n$ satisfying
\[
\mbox{{\rm \di\,}}\uu=f
\]
and
\begin{equation}
\label{cota con peso} \|\uu\|_{W^{1,p}(U,\omega)} \le
C\|f\|_{L^p(U,\omega)},
\end{equation}
with a constant $C$ depending only on $\w$, $U$, $p$ and $n$.
\end{theorem}

\begin{proof}
Using general results for singular integral operators we are going
to show that the explicit solution of $\di\uu=f$ introduced by
Bogovskii in \cite{B} (see also \cite{DM,G}) satisfies the desired
property. In what follows we consider $f$ defined in $\R^n$
extending it by zero to the complement of $U$.

Bogovskii's solution can be written as

\[
\uu(x)=\int_{U} \G(x,y)\,f(y)\,dy
\]
with $\G(x,y)=(G_1,\dots,G_n)$ given by
\[
\G(x,y)= \psi(y)\int_0^1
\frac{(x-y)}{s^{n+1}}\phi\left(y+\frac{x-y}{s}\right) \,ds,
\]
where $\varphi\in C_0^\infty(B)$ is such that $\int_B\varphi=1$
and $\psi\in C_0^{\infty}(\R^n)$ is a regularized characteristic
of $U$, i. e., $\psi(y)=1$ for any $y\in U$ and $\psi$ is
supported in a neighborhood of $U$.

In what follows the letter $C$ denotes a generic constant which
may depend on $n$, $p$, $\varphi$, $\w$, and the diameter of $U$,
that we will call $d$, but is independent of $f$ and $\uu$.

Let us first see that $\uu\in L^p(U,\w)^n$. It is known that (see
\cite{DM,G})

\begin{equation}
\label{DAp1} |\G(x,y)|\le \frac{C}{|x-y|^{n-1}}.
\end{equation}

Using (\ref{DAp1}) we have, for $x\in U$,
\begin{eqnarray*}
|\uu(x)|\,&\le&\,C\,\int_{\U} \frac{1}{\ \
|x-y|^{n-1}}\,|f(y)|\,dy\,\leq \,C\,\int_{B(x,d)}
\frac{1}{\ \ |x-y|^{n-1}}\,|f(y)|\,dy \\
&\leq&
\,C\,\sum_{k=0}^\infty\int_{\frac{d}{2^{k+1}}<|y-x|<\frac{d}{2^k}}
\frac{1}{\ \ |x-y|^{n-1}}\,|f(y)|\,dy\\
&\leq&
\,C\,\sum_{k=0}^\infty\int_{\frac{d}{2^{k+1}}<|y-x|<\frac{d}{2^k}}
\left(\frac{2^{k+1}}{d}\right)^{n-1}\,|f(y)|\,dy\\
&\le& \,C\,\sum_{k=0}^\infty
2^{-k}\,\frac{1}{|B(x,\frac{d}{2^k})|}\int_{B(x,\frac{d}{2^k})}
|f(y)|\,dy\,\le\,C\,Mf(x),
\end{eqnarray*}
where $Mf$ denotes the Hardy-Littlewood maximal function of $f$.
Since $\w\in A_p$, the maximal operator is bounded in
$L^p(\R^n,\w)$ (see for example \cite{Du,S2}), and therefore

\begin{equation}
\label{cota de u} \|\uu\|_{L^{p}(U,\w)^n}\le C\|f\|_{L^p(U,\w)}.
\end{equation}

Now, to see that the first derivatives of the components $u_j$ of
$\uu$ are also in $L^p(U,\w)$ we use that this derivatives can be
written in the following way (see \cite{DM,G}),

\[
\frac{\partial u_j}{\partial x_i} =\varphi_{ij}f + T_{ij}f,
\]
where $\varphi_{ij}$ is a function bounded by a constant depending
only on $\varphi$ and
\[
T_{ij}f(x)=\lim_{\ve\to 0}\int_{|y-x|>\ve} \frac{\partial
G_j}{\partial x_i}(x,y)f(y)\,dy.
\]
Therefore, to prove (\ref{cota con peso}) it only remains to prove
that the operators $T_{ij}$ are bounded in $L^p(\R^n,\w)$.

It was shown in \cite{DM,G} that $T_{ij}$ is continuous in
$L^p(\R^n)$ by using the Calder\'on-Zygmund singular integral
operator theory developed in \cite{CZ}.

We were not able to find in the literature that a general operator
of the form considered in \cite{CZ} is continuous in
$L^p(\R^n,\w)$ for $\w\in A_p$. However, such a continuity result
is known to hold for an operator of the form
\[
Tf(x)=\lim_{\ve\to 0}\int_{|y-x|>\ve} K(x,y)f(y)\,dy
\]
which is bounded in $L^2(\R^n)$ and with a kernel satisfying
\begin{equation}
\label{cota de K} |K(x,y)|\le \frac{C}{|x-y|^{n}},
\end{equation}
and the so called H\"ormander conditions, namely,
\[
|K(x,y)-K(x^\prime,y)| \le C\frac{|x-x^\prime|}{|x-y|^{n+1}}
\hspace*{1cm}{\rm if}\ |x-y|\ge 2|x-x^\prime|,
\]
and
\[
|K(x,y)-K(x,y^\prime)| \le C\frac{|y-y^\prime|}{|x-y|^{n+1}}
\hspace*{1cm}{\rm if}\ |x-y|\ge 2|y-y^\prime|
\]
see \cite[page 221]{S2}.

For $T_{ij}$ we have
\[
K(x,y) =\frac{\partial G_j}{\partial x_i}(x,y)
=\psi(y)\int_0^1\frac{\delta_{ij}}{s^{n+1}}\varphi\left(y+\frac{x-y}{s}\right)
+\frac{x_j-y_j}{s^{n+2}}\frac{\partial \varphi}{\partial x_i}
\left(y+\frac{x-y}{s}\right)\,ds,
\]
where $\delta_{ij}$ denotes the Kronecker symbol.

This kernel satisfies (\ref{cota de K}) (\cite{DM,G}) and also the
H\"ormander conditions (this was proved in \cite{N}).

In conclusion we obtain that, for any $i,j$,
\[
\left\|\frac{\partial u_j}{\partial x_i}\right\|_{L^p(U,\w)} \le C
\|f\|_{L^p(U,\w)}
\]
which together with (\ref{cota de u}) gives (\ref{cota con peso}).

To end the proof we have to show that $\uu$ vanishes at the
boundary, i.e., $\uu\in W^{1,p}_0(U,\w)^n$. For an arbitrary
weight this is not obvious from the definition of $\uu$. However,
once that we know the estimate (\ref{cota de u}) we can prove it
by density. We omit details because they are standard.

\end{proof}

\begin{remark}
If the weight $\w$ in the previous Theorem is a power of the
distance to the origin (which is one of the case of interest in
our applications to Stokes) it is not necessary to use the
H\"ormander conditions. Indeed, in this case (\ref{cota con peso})
can be proved using the results in (\cite{S1}).
\end{remark}

\section{Solutions of the divergence in cuspidal domains}
\label{sol de la div} \setcounter{equation}{0}

In this section we prove the existence of solutions of the
divergence in weighted Sobolev spaces for domains with an external
cusp.

We consider the following class of domains. Given integer numbers
$k\ge 1$ and $m\ge 0$ we define
\begin{eqnarray}
\label{ej} \O=\Big\{(x,y,z)\in I\times \R^k\times
I^m\,:\,|y|<x^\gamma\Big\}\subset\R^n,
\end{eqnarray}
where $n=m+k+1$, $I$ is the interval $(0,1)$ and $\gamma\ge 1$.

For $\gamma=1$, $\O$ is a convex domain while, for $\gamma>1$,
$\O$ has an external cusp. The set of singularities of the
boundary, which has dimension $m$, will be called $M$. Namely,
\begin{equation}
\label{def de M} M=\{{\bf 0}\}\times[0,1]^m\subset
\R^{k+1}\times\R^m.
\end{equation}

We will work with weighted Sobolev spaces where the weights are
powers of the distance to $M$ that will be called $d_M$.
Precisely, we will use the spaces $L^p(\O,d_M^{p\b})$ and
$W^{1,p}(\O,d_M^{p\b_1},d_M^{p\b_2})$ where $\b$, $\b_1$ and
$\b_2$ are real numbers. For $\b_1=\b_2=\b$ we will write
$W^{1,p}(\O,d_M^{p\b})$ instead of $W^1(\O,d_M^{p\b},d_M^{p\b})$.
It is well known that these spaces are Banach spaces (see
\cite{Ku}, Theorem 3.6. for details).

We consider $d_M$ defined everywhere in $\R^n$ and we are going to
use the following result that we state as a lemma for the sake of
clarity.

\begin{lemma}
\label{d esta en Ap} If $-(n-m)<\mu<(n-m)(p-1)$ then,
$d_M^{\mu}\in A_p$.
\end{lemma}
\begin{proof} It follows from the more general result proved
in Lemma 3.3 of \cite{DLg}. Indeed, calling $d_F$ the distance to
a compact set $F\subset\R^n$, it was proved in that paper that
$d_F^\mu\in A_p$ whenever the $m$-dimensional Hausdorff measure of
$B(x,r)\cap F$ is equivalent to $r^m$, for all $x\in F$ and
$r<diam(F)$.
\end{proof}

In what follows we will use several times that, for
$(x,y,z)\in\O$, $d_M(x,y,z)\simeq x$, where the symbol $\simeq$
denotes equivalence up to multiplicative constants. Indeed, it is
easy to see that $x\le d_M(x,y,z)=|(x,y)|\le (\sqrt{2})x$.

In the proof of the main result of this section we will use the
Hardy type inequality given in the next lemma.

\begin{lemma}
\label{Hardy} Let $\O$ be the domain defined in (\ref{ej}) and
$1<p<\infty$. Given $\kappa\in\R$, if $v\in
W_0^{1,p}(\O,d_M^{p\kappa})$ then, $v/x\in L^p(\O,d_M^{p\kappa})$
and there exists constant $C$, depending only on $p$ and $\kappa$,
such that
\begin{equation}
\label{hardy2} \left\|\frac{v}{x}\right\|_{L^p(\O,d_M^{p\kappa})}
\le C \left\|\frac{\partial v}{\partial
x}\right\|_{L^p(\O,d_M^{p\kappa})}.
\end{equation}
Consequently, $W_0^{1,p}(\O,d_M^{p\kappa})$ is continuously
imbedded in $W_0^{1,p}(\O,d_M^{p(\kappa-1)},d_M^{p\kappa})$.
\end{lemma}
\begin{proof} Let $\phi\in C_0^\infty(a,1)$, where $0 < a <1$. Then,

\begin{eqnarray}
\label{hardy1} \left(\int_a^1
\left|\phi(x)\right|^p\,x^{p\kappa-p}\right)^{1/p}\leq\frac{p}{|p\kappa-p+1|}
\left(\int_a^1\left|\phi^\prime(x)\right|^p\,x^{p\kappa}\right)^{1/p}.
\end{eqnarray}
Indeed, integrating by parts and applying the H\"older inequality
we have

\begin{eqnarray*}
\int_a^1
\left|\phi(x)\right|^p\,x^{p\kappa-p}&=&\frac{1}{p\kappa-p+1}\int_a^1
\left|\phi(x)\right|^p\,\left(x^{p\kappa-p+1}\right)^\prime\\
&\le&\frac{p}{p\kappa-p+1}\int_a^1
\left|\phi(x)\right|^{p-1}\left|\phi^\prime(x)\right|\,x^{p\kappa-p+1}\\
&\leq&\frac{p}{|p\kappa-p+1|}\left(\int_a^1
\left|\phi(x)\right|^p\,x^{p\kappa-p}\right)^{(p-1)/p}
\left(\int_a^1\left|\phi^\prime(x)\right|^p\,x^{p\kappa}\right)^{1/p}
\end{eqnarray*}
and dividing by $\left(\int_a^1
\left|\phi(x)\right|^p\,x^{p\kappa-p}\right)^{(p-1)/p}$  we obtain
(\ref{hardy1}).

Now, by a density argument it is enough to prove (\ref{hardy2})
for $v\in C_0^\infty(\O)$. Using (\ref{hardy1}) we have

\begin{eqnarray*}
\label{coi} \int_{\o}\left| v(x,y,z)\right|^p x^{p\kappa-p}\ {\rm
d}x\,{\rm d}y\,{\rm d}z &=&
\int_{I^m}\int_{|y|<1}\int_{|y|^{1/\g}}^1\left|
v(x,y,z)\right|^p x^{p\kappa-p}\ {\rm d}x\,{\rm d}y\,{\rm d}z\\
&\leq& C\int_{I^m}\int_{|y|<1}\int_{|y|^{1/\g}}^1\left|
\frac{\partial v(x,y,z)}{\partial x}\right|^p x^{p\kappa}\ {\rm
d}x\,{\rm d}y\,{\rm d}z\\
&=& C\int_{\o}\left|\frac{\partial v(x,y,z)}{\partial
x}\right|^px^{p\kappa}\ {\rm d}x\,{\rm d}y\,{\rm d}z.
\end{eqnarray*}

To conclude the proof we use that $d_M(x,y,z)\simeq x$ and
therefore, that $W_0^{1,p}(\O,d_M^{p\kappa})$ is continuously
imbedded in $W_0^{1,p}(\O,d_M^{p(\kappa-1)},d_M^{p\kappa})$
follows from (\ref{hardy2}).
\end{proof}

We can now prove the main result of this section.

\begin{theorem}
\label{teo} Let $\O$ be the domain defined in (\ref{ej}) for a
fixed $\g>1$, $M$ defined as in (\ref{def de M}), and
$1<p<\infty$. If
$\beta\in\left(\frac{-\gamma(n-m)}{p}-\frac{\gamma-1}{p^\prime},
\frac{\gamma(n-m)}{p^\prime}-\frac{\gamma-1}{p^\prime}\right)$ and
$\eta\in\R$ is such that $\eta\ge\beta+\gamma-1$ then, given $f\in
L_0^p(\O,d_M^{p\b})$, there exists $\uu\in
W_0^{1,p}(\O,d_M^{p(\eta-1)},d_M^{p\eta})^n$ satisfying
\begin{equation}
\label{solucion en cuspide} \mbox{{\rm \di\,}}\uu=f
\end{equation}
and
\begin{equation}
\label{pesoalaizq}
\|\uu\|_{W^{1,p}(\O,d_M^{p(\eta-1)},d_M^{p\eta})^n} \le
C\|f\|_{L^p(\O,d_M^{p\b})}
\end{equation}
with a constant $C$ depending only on $\gamma$, $\b$, $\eta$, $p$
and $n$.
\end{theorem}

\begin{proof}
It is enough to prove the result for the case $\eta=\b+\gamma-1$.
Therefore we are going to consider this case.

Define
\begin{equation}
\label{omega sombrero} \ho=\Big\{(\hx,\hy,\hz)\in I\times
\R^k\times I^m\,:\,|\hy|<\hx\Big\}\subset\R^n
\end{equation}
and let $F:\ho\rightarrow\O$ be the one-to-one application given
by
\[
F(\hx,\hy,\hz)=(\hx^\alpha,\hy,\hz)=(x,y,z),
\]
where $\alpha=1/\gamma$.

By this change of variables we associate functions defined in $\O$
with functions defined in $\ho$ in the following way,
\[
h(x,y,z)=\hat{h}(\hx,\hy,\hz).
\]
Now, for $f\in L^p_0(\O,d_M^{p\b})$, we define $\hg:\ho\to\O$ by
\[
\hat{g}(\hx,\hy,\hz):=\alpha \hx^{\alpha-1}\hat{f}(\hx,\hy,\hz).
\]
We want to apply Theorem \ref{divAp} for $\hat{g}$ on the convex
domain $\ho$ and then obtain the desired solution of
(\ref{solucion en cuspide}) by using the so called Piola transform
for vector fields.

In the rest of the proof we will use several times that, for
$(x,y,z)\in\O$, $d_M(x,y,z)\simeq x$, $\det DF(\hx,\hy,\hz)=\a
\hx^{\a-1}$ and $\det DF^{-1}(x,y,z)=\g x^{\g-1}$.

First let us see that, for
$\hat{\beta}=\alpha\left(\beta+(\gamma-1)/p^\prime\right)$, we
have
\begin{eqnarray}
\label{condiciondeg} \hat{g}\in
L^p_0(\ho,d_M^{p\hb})\hspace*{1cm}{\rm and}\hspace{1cm}
\|\hat{g}\|_{L^p(\ho,d_M^{p\hb})}\simeq\|f\|_{L^p(\O,d_M^{p\b})}.
\end{eqnarray}
Indeed, we have

\begin{eqnarray*}
\|\hat{g}\|_{L^p(\ho,d_M^{p\hb})}^p&\simeq
&\int_{\ho}|\hat{g}|^p\hx^{p\hb}=\alpha^p
\int_{\ho}|\hat{f}|^p\hx^{p(\a-1)}\hx^{\a p(\b+(\gamma-1)/p^\prime)}\\
&=&\alpha^p\int_{\o}|f|^px^{ p\b+1-\gamma}\gamma
x^{\gamma-1}\simeq\|f\|_{L^p(\O,d_M^{p\b})}^p
\end{eqnarray*}
and
\begin{eqnarray*}
\int_{\ho} \hat{g}=\alpha\int_{\ho} \hat{f}\hx^{\a-1}=
\alpha\int_{\o} fx^{1-\gamma}\gamma x^{\gamma-1}=\int_{\o} f=0.
\end{eqnarray*}
Thus, (\ref{condiciondeg}) holds.

Observe that, from Lemma (\ref{d esta en Ap}) and our hypothesis
on $\b$, we have $d_M^{p\hb}\in A_p$. In particular, it follows
from Remark \ref{integrable} that $\hat g\in L^1(\ho)$  and
therefore the mean value of $f$ in $\O$ is well defined.

Now, from Theorem \ref{divAp} we know that there exists
$\hat{\vv}\in W_0^{1,p}(\ho,d_M^{p\hb})^n$ such that
\begin{equation}
\label{divAp2} \di\hat{\vv}=\hat{g}
\end{equation}
and
\begin{equation}
\label{divAp3} \|\hat{\vv}\|_{W^{1,p}(\ho,d_M^{p\hb})^n} \le
C\|\hat{g}\|_{L^p(\ho,d_M^{p\hb})}.
\end{equation}
Now, we define $\uu$ as the Piola transform of $\hat\vv$, namely,
\[
\uu(x,y,z)=\frac{1}{\det DF}DF(\hx,\hy,\hz)\hat{\vv}(\hx,\hy,\hz)
\]
or equivalently, if $\hat\vv=(\hat v_1,\dots,\hat v_n)$,
\[
\uu(x,y,z)=\gamma x^{\gamma-1}\left(\alpha x^{1-\gamma} \hat
v_1(x^\gamma,y,z),\hat v_2(x^\gamma,y,z),\dots,\hat
v_n(x^\gamma,y,z)\right).
\]
Then, using (\ref{divAp2}), it is easy to see that
\[
\di\uu=f.
\]
To prove (\ref{pesoalaizq}) we first show that

\begin{equation}
\label{pesoalaizquierda2} \|\uu\|_{W^{1,p}(\O,d_M^{p\eta})^n} \le
C\|f\|_{L^p(\O,d_M^{p\b})}.
\end{equation}
In view of the equivalence of norms given in (\ref{condiciondeg})
and the estimate (\ref{divAp3}), to prove
(\ref{pesoalaizquierda2})  it is enough to see that
\begin{equation}
\label{campo} \|\uu\|_{W^{1,p}(\O,d_M^{p\eta})^n} \le
C\|\hat{\vv}\|_{W^{1,p}(\ho,d_M^{p\hat{\beta}})^n}.
\end{equation}
But, we have

\begin{equation}
\label{cotadeu1} \|u_1\|_{L^p(\O,d_M^{p\eta})}^p \simeq\int_\O
|u_1|^px^{p\eta} =\a\int_{\ho} |\hat v_1|^p\hx^{\a
p\eta}\hx^{\a-1} \simeq\|\hat v_1\|_{L^p(\O,d_M^{p\hb})}^p,
\end{equation}
where in the last step we have used $\a p\eta+\a-1=p\hb$. In an
analogous way we can show that, for $j=2,\dots,n$,
$$
\|u_j\|_{L^p(\O,d_M^{p\eta})} \le C \|\hat
v_j\|_{L^p(\O,d_M^{p\hb})}.
$$
Then, it only remains to bound the derivatives of the components
of $\uu$. That
\[
\left\|\frac{\partial u_1}{\partial
y_1}\right\|_{L^p(\O,d_M^{p\eta})}^p \simeq \left\|\frac{\partial
\hat{v}_1}{\partial\hy_1}\right\|_{L^p(\ho,d_M^{p\hb})}^p
\]
follows exactly as (\ref{cotadeu1}). Let us now estimate
$\frac{\partial u_2}{\partial x}$. Using

\[
\left|\frac{\partial u_2}{\partial x}\right|
=\gamma^2\left|\frac{\gamma-1}{\gamma}\frac{\hat{v}_2(x^\gamma,y,z)}
{x^\gamma}+\frac{\partial\hat{v}_2(x^\gamma,y,z)}{\partial
\hx}\right|x^{2(\gamma-1)}
\]
and Lemma \ref{Hardy} for $\ho$ we have

\begin{eqnarray*}
\left\|\frac{\partial u_2}{\partial
x}\right\|_{L^p(\O,d_M^{p\eta})}^p
&\simeq&\int_{\O}\left|\frac{\partial u_2}{\partial
x}\right|^px^{p\eta} \le
C\int_{\ho}\left(\left|\frac{\hat{v}_2}{\hx}\right|^p
+\left|\frac{\partial\hat{v}_2}{\partial\hx}\right|^p\right)
\hx^{2p(1-\alpha)}\hx^{\a p\eta+\a-1}\\
&\le& C\int_{\ho}\left(\left|\frac{\hat{v}_2}{\hx}\right|^p
+\left|\frac{\partial\hat{v}_2}{\partial
\hx}\right|^p\right)\hx^{p\hb} \le C\int_{\ho}\left|\frac{\partial
\hat{v}_2}{\partial\hx}\right|^p\hx^{p\hb}
=\left\|\frac{\partial\hat{v}_2}{\partial
\hx}\right\|_{L^p(\ho,d_M^{p\hb})}^p,
\end{eqnarray*}
where we have used again $\a p\eta+\a-1=p\hb$ and that
$2p(1-\alpha)>0$.

All the other derivatives of the components of $\uu$ can be
bounded in an analogous way and therefore (\ref{campo}) holds.

Now, since
$$
\uu|_{\partial\o}=\frac{1}{\det DF}DF\hat{\vv}|_{\partial\ho},
$$
it is easy to check that $\uu$ belongs to the closure of
$C_0^\infty(\O)^n$,i. e., $\uu\in W_0^{1,p}(\O,d_M^{p\eta})^n$ and
by Lemma \ref{Hardy} $\uu\in
W_0^{1,p}(\O,d_M^{p(\eta-1)},d_M^{p\eta})^n$ as we wanted to show.
\end{proof}

\begin{remark} The hypothesis that
$\b<\frac{\gamma(n-m)}{p^\prime}-\frac{\gamma-1}{p^\prime}$ is
necessary in order to have the condition $\int_\O f=0$ well
defined for $f\in L^p(\O,d_M^{p\b})$. Indeed, if
$\b\ge\frac{\gamma(n-m)}{p^\prime}-\frac{\gamma-1}{p^\prime}$, it
is easy to check that $f(x,y,z)=(1-\log x)^{-1}x^{\g-1-\g(n-m)}$
belongs to $L^p(\O,d_M^{p\b}) \setminus L^1(\O)$.
\end{remark}

\begin{remark} It can be shown that the condition $\eta\ge\beta+\gamma-1$
assumed in the theorem is also necessary. Indeed, if
$\eta-\beta<\gamma-1$, it can be shown by generalizations of the
example presented in Section \ref{section2}, that there exists
$f\in L_0^p(\O,d_M^{p\b})$ such that a solution $\uu$ of
(\ref{solucion en cuspide}) satisfying (\ref{pesoalaizq}) does not
exist (see \cite{ADLg} for the details).
\end{remark}

\section{Application to the Stokes equations}
\setcounter{equation}{0} \label{aplicacion1}

In this section we show how the results obtained in the previous
section can be applied to prove the well posedness of the Stokes
equations in appropriate weighted Sobolev spaces for cuspidal
domains.

Indeed, combining the variational analysis given in Sections
\ref{section2} with the results in Section \ref{sol de la div} we
obtain the following theorem.

\begin{theorem}
\label{teo5.1} Given $\g\ge 1$, let $\O$ be the domain defined in
(\ref{ej}). If $\ff\in H^{-1}(\O)^n$ then, there exists a unique
$(\uu,p)\in H^1_0(\O)^n \times L^2(\O,d_M^{2(\g-1)})$, with $p$
satisfying $\int_\O p\,d_M^{2(\g-1)}=0$, weak solution of the
Stokes equations (\ref{stokes1}). Moreover,
$$
\|\uu\|_{H^1_0(\O)^n} + \|p\|_{L^2(\O,d_M^{2(\g-1)})} \le C
\|\ff\|_{H^{-1}(\O)^n}
$$
with a constant $C$ depending only on $\g$ and $n$.
\end{theorem}
\begin{proof}
Consider the particular case $\eta=0$, $\b=1-\g$ and $p=2$ in
Theorem \ref{teo}. It is easy to check that in this case $\b$
satisfies the hypothesis of that theorem for any values of $n$ and
$m$ (recall that $m\le n-2$), i. e,
$$
\b=1-\g\in\left(\frac{-\g(n-m)}{2}-\frac{\g-1}{2},
\frac{\gamma(n-m)}{2}-\frac{\g-1}{2}\right).
$$
Then, given $f\in L_0^2(\O,d_M^{2(1-\g)})$ there exists $\uu\in
H_0^1(\O)^n$ satisfying
$$
\mbox{{\rm \di\,}}\uu=f
$$
and
$$
\|\uu\|_{H^1(\O)^n} \le C\|f\|_{L^2(\O,d_M^{2(1-\g)})}
$$
with a constant $C$ depending only on $\gamma$ and $n$. Therefore,
the result follows immediately from Theorem \ref{existencia y
unicidad}.
\end{proof}

In the next corollary we show the well posedness of the Stokes
equations in standard spaces.

\begin{corollary}
Given $\g\ge 1$, let $\O$ be the domain defined in (\ref{ej}) and
$\ff\in H^{-1}(\O)^n$. If $r_0$ is defined by
$$
r_0=2-\frac{4(\gamma-1)}{\gamma(k+2)-1}.
$$
Then, $r_0>0$, and for $0< r < r_0$, there exists a unique
$(\uu,p)\in H^1_0(\O)^n \times L^r(\O)$, with $p$ satisfying
$\int_\O p\,d_M^{2(\g-1)}=0$, weak solution of the Stokes
equations (\ref{stokes1}). Moreover, there exists a constant $C$
depending only on $n$, $\g$ and $r$ such that
$$
\|\uu\|_{H^1_0(\O)^n} +\|p\|_{L^r(\O)} \le C
\|\ff\|_{H^{-1}(\O)^n}.
$$
In particular, if $k\ge 2$, or $k=1$ and $\g<3$, $p\in L^1(\O)$.
\end{corollary}

\begin{proof} Since $\g>1$ and $k\ge 1$ it follows that $r_0>0$.
Now, given a positive $r<r_0$ it is enough to see that, if
$(\uu,p)$ is the solution given by Theorem \ref{teo5.1}, then
$p\in L^r(\O)$ and
\begin{equation}
\label{cotapLr} \|p\|_{L^r(\O)}\le C
\|p\|_{L^2(\O,d_M^{2(\g-1)})}.
\end{equation}

It is easy to see that $\int_\O d_M^s < +\infty$ for any $s>-\g
k-1$. Then, applying the H\"older inequality with $2/r$ and its
dual exponent we have
$$
\|p\|^r_{L^r(\O)}=\int_\O |p|^r d_M^{(\g-1)r} d_M^{(1-\g)r} \le
\|p\|^r_{L^2(\O,d_M^{2(\g-1)})} \left(\int_\O
d_M^{\frac{2(1-\g)r}{2-r}}\right)^{\frac{2-r}2}.
$$
Since $r < r_0$, we have $(2(1-\gamma)r)/(2-r)>-\gamma k-1$, and
so the integral on the right hand side is finite. Therefore,
(\ref{cotapLr}) is proved. Finally, if $k\ge 2$, or $k=1$ and
$\g<3$, it is easy to check that $r_0>1$ and therefore $p\in
L^1(\O)$.
\end{proof}

To end this section let us show the results of the above theorem
and corollary in the particular cases $n=2$ and $n=3$. We will use
here the usual notation $x=(x_1,x_2)\in\R^2$ or
$x=(x_1,x_2,x_3)\in\R^3$.

For $n=2$ we have $m=0$ and, for $\g\ge 1$, the domain is
$$
\O=\Big\{x=(x_1,x_2)\in \R^2:\, 0<x_1<1 \, , \,
|x_2|<x_1^\g\Big\}.
$$
In this case $M=(0,0)$ and therefore $d_M(x)=|x|$. Then, for
$\ff\in H^{-1}(\O)^2$, there exists a unique
$$
(\uu,p)\in H^1_0(\O)^2 \times L^2(\O,|x|^{2(\g-1)}),
$$
with $p$ satisfying $\int_\O p\,|x|^{2(\g-1)}=0$, weak solution of
the Stokes equations . Moreover,
\begin{equation}
\label{cota en R2} \|\uu\|_{H^1_0(\O)^2} +
\|p\|_{L^2(\O,|x|^{2(\g-1)})} \le C \|\ff\|_{H^{-1}(\O)^2}
\end{equation}
and, for $r<2-\frac{4(\g-1)}{3\g-1}$,
\begin{equation}
\label{cota en R2 2} \|p\|_{L^r(\O)}\le C\|\ff\|_{H^{-1}(\O)^2}
\end{equation}
with a constant $C$ depending only on $\g$ and $r$.

For $n=3$ we have the two possible cases $m=0$ or $m=1$. In the
first case the domain has a cuspidal point and is given by

$$
\O=\Big\{x=(x_1,x_2,x_3)\in \R^3:\, 0<x_1<1 \, , \,
\sqrt{x^2_2+x^2_3}<x_1^\g\Big\}.
$$
In this case we obtain exactly the same estimates (\ref{cota en
R2}) and (\ref{cota en R2 2}) with obvious changes of dimension.
The only difference is that now $r<2-\frac{4(\g-1)}{4\g-1}$.
Observe that in particular, in this case $p\in L^1(\O)$.

Finally, when $m=1$, the domain has a cuspidal edge and is given
by
$$
\O=\Big\{x=(x_1,x_2,x_3)\in \R^3:\, 0<x_1<1 \, , \, 0<x_3<1 \,  ,
\, |x_2|<x_1^\g\Big\}
$$
and, defining $\bar x=(x_1,x_2)$, we have $d_M(x)=|\bar x|$ and
the a priori estimates
$$
\|\uu\|_{H^1_0(\O)^3} + \|p\|_{L^2(\O,|\bar x|^{2(\g-1)})} \le C
\|\ff\|_{H^{-1}(\O)^3}
$$
and, for $r<2-\frac{4(\g-1)}{3\g-1}$,
$$
\|p\|_{L^r(\O)}\le C\|\ff\|_{H^{-1}(\O)^3}.
$$

\section{Weighted Korn type inequalities}
\label{aplicacion2} \setcounter{equation}{0}

Important and well-known consequences of the existence of a right
inverse of the divergence operator in Sobolev spaces are the
different cases of Korn inequalities. It is also known that the
classic first and second cases (in the terminology introduced by
Korn) can be derived from the following inequality,

\begin{equation}
\label{des de Korn} \|D\vv\|_{L^p(\O)^{n\times n}} \le
C\left\{\|\vv\|_{L^p(\O)^n} +\|\ve(\vv)\|_{L^p(\O)^{n\times
n}}\right\},
\end{equation}
where we are using the usual notation for the symmetric part of
the differential matrix $D\vv$ of a vector field
$(v_1,\dots,v_n)$, namely,
$$
\ve_{ij}(\vv)=\frac{1}{2}\left(\frac{\partial v_i}{\partial x_j}
+\frac{\partial v_j}{\partial x_i}\right).
$$
For the cuspidal domains that we are considering this inequality
is not valid (counterexamples are given in \cite{ADLg,GG,W}). In
view of our results on solutions of the divergence it is natural
to look for Korn type inequalities in weighted Sobolev spaces. For
general H\"older $\alpha$ domains, inequalities of this kind were
obtained in \cite{ADL} using weights which are powers of the
distance to the boundary. Here we are interested in stronger
results for the particular class of H\"older $\alpha$ domains
defined in (\ref{ej}). We are going to prove estimates in norms
involving the distance to the cusp.

It is not straightforward to generalize the classic arguments to
derive Korn inequalities from the existence of right inverses of
the divergence to the weighted case. We do not know how to do it
if we work with weighted norms in both sides of the inequality
(\ref{des de Korn}). Therefore, we are going to prove a result for
a general weight and afterwards, we will obtain more general
inequalities for the case of weights which are powers of the
distance to the cusp, using an argument introduced in \cite{BK}.

Let us mention that in what follows we state and prove several
inequalities assuming that the left hand side is finite.
Afterwards, by density arguments, one can conclude that these
inequalities are valid whenever the right hand side is finite.
This is a usual procedure.

Given $1<p<\infty$, a domain $U\subset\R^n$, and a weight $\w$, we
denote with $W^{-1,p'}(U,\w^{1-p'})$ the dual space of
$W^{1,p}_0(U,\w)$. Observe that
$W^{-1,p}(U,\w)=W^{1,p'}_0(U,\w^{1-p'})$.

\begin{lemma}
\label{Lions} Given a weight $\w$, a bounded domain
$U\subset\R^n$, and $1<p<\infty$, assume that for any $g\in
L^{p'}_0(U)$  there exists $\uu\in W_0^{1,p'}(U,\w^{1-p'})^n$ such
that ${\rm\di}\uu=g$ and
$$
\|\uu\|_{W^{1,p'}(U,\w^{1-p'})^n}\le C \|g\|_{L^{p'}(U)},
$$
with a constant $C$ depending only on $U$, $p$, and $\w$. Fix an
open ball $B\subset U$. Then, for any $f\in L^p(U)$,
$$
\|f\|_{L^p(U)}\le\ C\left\{\|f\|_{W^{-1,p}(B)} +\|\nabla
f\|_{W^{-1,p}(U,\w)^n}\right\},
$$
where the constant $C$ depends only on $U$, $B$, $p$, and $\w$.
\end{lemma}
\begin{proof} Take $f\in L^p(U)$.
If  $\bar{f}$ denotes the mean value of $f$ over $U$ we have, for
$g\in L^{p'}(U)$,
$$
\int_\U (f-\bar f)g=\int_\U (f-\bar f)(g-\bar g).
$$
But, from our hypothesis,  there exists a solution $\uu\in
W_0^{1,p'}(U,\w^{1-p'})^n$ of $\di\uu =g-\bar g$ satisfying
$$
\|\uu\|_{W^{1,p'}(U,\w^{1-p'})^n}\le C \|g-\bar g\|_{L^{p'}(U)}.
$$
Thus,
\begin{eqnarray*}
\int_\U (f-\bar f)g=\int_U (f-\bar f)\di\uu
&\le& \|\nabla f\|_{W^{-1,p}(U,\w)^n}\|\uu\|_{W^{1,p'}(U,\w^{1-p'})^n}\\
&\le& C\|\nabla f\|_{W^{-1,p}(U,\w)^n}\|g-\bar g\|_{L^{p'}(U)}.
\end{eqnarray*}
Therefore,  by duality,
\begin{equation}
\label{f menos prom} \|f-\bar f\|_{L^p(U)}\le C\|\nabla
f\|_{W^{-1,p}(U,\w)^n} .
\end{equation}
Now, we decompose $f$ as
$$
f=\left(f-f_\varphi\right)+f_\varphi,
$$
where $f_\varphi:=\int_B f\varphi$ with $\varphi\in C_0^\infty(B)$
such that $\int_B\varphi=1$. Thus,
\[
f-f_\varphi=f-\bar{f} +\int_B\left(\bar{f}-f\right)\varphi,
\]
and so, using (\ref{f menos prom}),
\begin{eqnarray*}
\|f-f_\varphi\|_{L^p(U)}\le\left(1+\|\varphi\|_{L^{p'}(B)}\right)
\|f-\bar{f}\|_{L^p(U)} \le C\|\nabla f\|_{W^{-1,p}(U,\w)^n}.
\end{eqnarray*}
Therefore, to conclude the proof we have to estimate
$\|f_\varphi\|_{L^p(U)}$. But,
$$
\|f_\varphi\|_{L^p(U)}\leq |U|^{1/p}\left|\int_B f\varphi\right|
\le |U|^{1/p}\|f\|_{W^{-1,p}(B)}\|\varphi\|_{W^{1,p'}_0(B)}.
$$
\end{proof}

Using this lemma we can generalize a classic argument to prove a
Korn type inequality obtaining the following result.

\begin{theorem}
\label{Kornsimple} Given a weight $\w$, a bounded domain
$U\subset\R^n$, and $1<p<\infty$, assume that for any $g\in
L^{p'}_0(U)$  there exists $\uu\in W_0^{1,p'}(U,\w^{1-p'})^n$ such
that ${\rm\di}\uu=g$ and
$$
\|\uu\|_{W^{1,p'}(U,\w^{1-p'})^n}\le C \|g\|_{L^{p'}(U)},
$$
with a constant $C$ depending only on $U$, $p$, and $\w$. Fix an
open ball $B\subset U$. Then, for any $\vv\in W^{1,p}(U)^n$,
$$
\|D\vv\|_{L^p(U)^{n\times n}} \le C\left\{\|\vv\|_{L^p(B)^n}
+\|\ve(\vv)\|_{L^p(U,\w)^{n\times n}}\right\},
$$
where the constant $C$ depends only on $U$, $B$, $p$, and $\w$.
\end{theorem}
\begin{proof}
It is known that, for any $g\in L^p(B)$,
\begin{equation}
\label{cota de g1} \left\|\frac{\partial g}{\partial
x_j}\right\|_{W^{-1,p}(B)} \le \|g\|_{L^p(B)}.
\end{equation}
Analogously, for any $g\in L^p(U,\w)$, we have
\begin{equation}
\label{cota de g2} \left\|\frac{\partial g}{\partial
x_j}\right\|_{W^{-1,p}(U,\w)} =\sup_{0\neq\phi\in
W_0^{1,p'}(U,\w^{1-p'})} \frac{\left|\int_U
g\frac{\partial\phi}{\partial x_j}\right|}
{\|\phi\|_{W^{1,p'}(U,\w^{1-p'})}} \le \|g\|_{L^p(U,\w)}.
\end{equation}
On the other hand, applying Lemma \ref{Lions}, we have
$$
\left\|\frac{\partial v_i}{\partial x_j}\right\|_{L^p(U)} \le\
C\left\{\left\|\frac{\partial v_i}{\partial
x_j}\right\|_{W^{-1,p}(B)} +\left\|\nabla\frac{\partial
v_i}{\partial x_j}\right\|_{W^{-1,p}(U,\w)^n}\right\}.
$$
Using now the well known identity
$$
\frac{\partial^2 v_i}{\partial x_j\partial x_k}
=\frac{\partial\ve_{ik}(\vv)}{\partial x_j}
+\frac{\partial\ve_{ij}(\vv)}{\partial x_k}
-\frac{\partial\ve_{jk}(\vv)}{\partial x_i}
$$
in the last term on the right hand side, and the inequalities
(\ref{cota de g1}) and (\ref{cota de g2}), we conclude the proof.
\end{proof}

An immediate consequence of Theorems \ref{Kornsimple} and
\ref{teo} is the following.

\begin{corollary}
Given $\g\ge 1$, let $\O$ be the domain defined in (\ref{ej}), $M$
defined in (\ref{def de M}), $1<p<\infty$, and $B\subset\O$ an
open ball. Then, there exists a constant $C$, which depends only
on $\O$, $B$, and $p$, such that for all $\uu\in W^{1,p}(\O)$,
$$
\|D\uu\|_{L^p(\O)^{n\times n}} \le C\left\{\|\uu\|_{L^p(B)^n}
+\|\ve(\uu)\|_{L^p(\O,d_M^{p(1-\g)})^{n\times n}}\right\}.
$$
\end{corollary}
\begin{proof} According to Theorem \ref{teo}, for any
$g\in L^{p'}_0(\O)$  there exists $\uu\in
W_0^{1,p'}(\O,d_M^{p'(\g-1)})^n$ such that ${\rm\di}\uu=g$ and
$$
\|\uu\|_{W^{1,p'}(\O,d_M^{p'(\g-1)})^n}\le C \|g\|_{L^{p'}(\O)},
$$
with a constant $C$ depending only on $\g$ and $p$. Therefore,
Theorem \ref{Kornsimple} applies for $\w=d_M^{p(1-\g)}$.
\end{proof}

We conclude the paper proving more general Korn type inequalities
for the cuspidal domains defined in (\ref{ej}). To obtain these
inequalities we use an argument introduced in \cite{BK}.

\begin{theorem}
\label{Korn} Given $\g\ge 1$, let $\O$ be the domain defined in
(\ref{ej}), $M$ defined in (\ref{def de M}), $1<p<\infty$,
$B\subset\O$ an open ball, and $\b\ge 0$. Then, there exists a
constant $C$, which depends only on $\O$, $B$, $p$, and $\b$, such
that for all $\uu\in W^{1,p}(\O,d_M^{p\b})$
$$
\|D\uu\|_{L^p(\O,d_M^{p\b})^{n\times n}} \le
C\left\{\|\uu\|_{L^p(B)^n}
+\|\ve(\uu)\|_{L^p(\O,d_M^{p(\b+1-\g)})^{n\times n}}\right\}.
$$
\end{theorem}

\begin{proof} To simplify the notation we will assume that
$m=0$ in the definition of $\O$. The other cases can be treated
analogously.

Let $n^\prime\in\NN_0$ and $0<s\leq\gamma$ be such that
$sn^\prime=p\beta$. As in \cite{BK} we introduce
\begin{equation}
\label{ejgeneral}
\O^{n^\prime,s}=\{(x,y,\zp)\in\R^{n+n^\prime}\,:\,(x,y)\in\o,\,\zp\in\R^{n^\prime}\
{\rm with}\ |\zp|<x^s\}.
\end{equation}
Suppose that the hypothesis in Theorem \ref{Kornsimple} on
solutions of the divergence is verified for $U=\oj$ and
$\w=x^{p(1-\g)}$. Then, if $B'\subset\oj$ is a ball with the same
radius and center than $B$, we have
\begin{eqnarray}
\label{Buckleykoskela} \|D\vv\|_{L^p(\oj)^{(n+n')\times (n+n')}}
\le C\left\{\|\vv\|_{L^p(B')^{n+n'}}+
\|\ve(\vv)\|_{L^p(\oj,x^{p(1-\g)})^{(n+n')\times (n+n')}}\right\},
\end{eqnarray}
for all $\vv\in W^{1,p}(\oj)^{n+n^\prime}$.

Now, given $\uu$ in $W^{1,p}(\O,d_M^{p\b})^n$ we define
$$
\vv(x,y,\zp)=(\uu(x,y),\underbrace{0,\ldots,0}_{n^\prime}).
$$
Then, using that for $(x,y)\in\O$, $d_M(x,y)\simeq x$, it is easy
to check that (\ref{Buckleykoskela}) is equivalent to
$$
\|D\uu\|_{L^p(\O,d_M^{p\b})^{n\times n}} \le C
\left\{\|\uu\|_{L^p(B)^n}
+\|\ve(\uu)\|_{L^p(\O,d_M^{p(\b+1-\g)})^{n\times n}}\right\}.
$$
Hence, to finish the proof we have to verify the hypothesis of
Theorem \ref{Kornsimple} for the domain $\oj$ with the weight
$\w=x^{p(1-\g)}$. Since in this case $\w^{1-p'}=x^{p'(\g-1)}$, we
have to show that, for any $g\in L^{p'}_0(\oj)$, there exists
$\ww\in W_0^{1,p'}(\oj,x^{p'(\g-1)})^n$ such that ${\rm\di}\ww=g$
and
$$
\|\ww\|_{W^{1,p'}(\oj,x^{p'(\g-1)})^n} \le C \|g\|_{L^{p'}(\oj)}.
$$
But this can be proved exactly as Theorem \ref{teo}, using now the
convex domain
\begin{eqnarray*}
\ho^{n^\prime,s}:=\{(\hx,\hy,\hzp)\in\R^{n+n^\prime}\,:\,(\hx,\hy)\in\ho,\,\zp\in\R^{n^\prime}\
{\rm with}\ |\zp|<x^{\alpha s}\},
\end{eqnarray*}
with $\ho$ defined as in (\ref{omega sombrero}), and the
one-to-one map $F:\hoj\to \oj$ defined by
\[
F(\hx,\hy,\hzp):=(\hx^\alpha,\hy,\hzp).
\]
\end{proof}

\end{document}